\documentclass[11pt]{article}
\usepackage{amsmath,amssymb,amsthm}
\usepackage[colorlinks,citecolor=blue]{hyperref}
\usepackage[utf8]{inputenc}
\usepackage{tikz}
\usetikzlibrary{shapes.geometric,calc}
\usepackage[all]{xy}
\usepackage{rotating}
\usepackage[english]{babel}

\newtheorem{theorem}{Theorem}[section] 
\newtheorem{proposition}[theorem]{Proposition} 
\newtheorem{conjecture}[theorem]{Conjecture} 
\newtheorem{corollary}[theorem]{Corollary} 
\newtheorem{lemma}[theorem]{Lemma}

\theoremstyle{definition}
\newtheorem{definition}[theorem]{Definition}
\newtheorem{remark}[theorem]{Remark} 
\newtheorem{example}[theorem]{Example}

\newcommand{\equic}[2][1 cm]{
\draw (0,0) circle (#1);
\pgfmathparse{#1/1 cm+0.25};
\edef\oc{\pgfmathresult cm};
  \foreach \i in {2,...,#2} {
    \pgfmathtruncatemacro\result{\i-1}
    \coordinate (N\result) at (-\i*360/#2:#1);
    \fill[black] (N\result) circle (0.05 cm);
    \draw (-\i*360/#2:\oc) node{$\result$};
}
\pgfmathtruncatemacro\r{#2}
\coordinate (N\r) at (-360/#2:#1);
\fill[black] (N\r) circle (0.05 cm);
\draw (-360/#2:\oc) node{$\r$};
}

\newcommand{\kk}{\mathbf{k}}

\newcommand{\Hom}{\operatorname{Hom}}

\newcommand{\Tam}{\operatorname{Tam}}

\newcommand{\Inc}{\operatorname{Inc}}
\newcommand{\Dec}{\operatorname{Dec}}
\newcommand{\Po}{\mathcal{P}}
\newcommand{\Se}{\mathbb{S}}

\newcommand{\modu}{\operatorname{-mod}}
\newcommand{\Rpk}{A_\kk(P)\modu}
\newcommand{\RTk}{A_\kk(\mathrm{Tam}_n)}
\newcommand{\NCT}{\mathbf{NCP}}
\title{The bounded derived categories of the Tamari lattices are fractionally Calabi-Yau}
\author{Baptiste Rognerud}
\date{}

\usepackage{array}
\newcolumntype{P}[1]{>{\centering\arraybackslash}p{#1}}

\begin{document}

\maketitle


\begin{abstract}
We prove that the bounded derived category of the incidence algebra of the Tamari lattice is fractionally Calabi-Yau, giving a positive answer to a conjecture of Chapoton. The proof involves a combinatorial description of the Serre functor of this derived category on a sufficiently nice family of indecomposable objects.  
 \end{abstract}
 \section{Introduction}
 
 \emph{Tamari lattices} are partial orders on Catalan sets introduced by Dov Tamari in his thesis. They are now classical and well studied objects in combinatorics and they appear in an incredible number of recent developments in mathematics such as, algebra, computer science, category theory, topology and many others. 
 
 In representation theory, they have two classical interpretations: as posets of \emph{tilting modules} of an equioriented quiver of type $A$ (\cite{buan_krause,happel_unger}) and as posets of \emph{cluster-tiling} modules for the same quiver. In particular, the Tamari lattices are part of the \emph{Cambrian lattices} of type $A$ in the sense of Reading (\cite{reading}).
 
 Chapoton was one of the first to realize that the representation theory of these lattices is also extremely fascinating. At first he used the Auslander-Reiten translation of the Tamari lattices to describe an anti-cyclic structure on the Dendriform operad (\cite{chapoton_op,chapoton_coxeter_tamari}). Then, using this anti-cyclic structure he proved that the so-called \emph{Coxeter transformation} of the Tamari lattices is \emph{periodic} (\cite{chapoton_coxeter_tamari}). 
 
 A couple of years later he published a beautiful article (\cite{chapoton_derived_tamari}) which contains three important conjectures on the derived category of the incidence algebras of these lattices. The main conjecture can be seen as a `categorification' of his result on the periodicity of the Coxeter transformation.  
 
 \begin{conjecture}[Chapoton]\label{conj_A} Let $n\in \mathbb{N}$ and $\Tam_n$ be the Tamari lattice of binary trees with $n$ inner vertices. Let $\kk$ be a field. Then, the bounded derived category of the incidence algebra of $\Tam_n$ over $\kk$ is fractionally Calabi-Yau of dimension $\big(n(n-1),2n+2\big)$. 
 \end{conjecture}

 The notion of \emph{Calabi-Yau} category has its roots in algebraic geometry: it generalizes elliptic curves, abelian varieties and $K3$ surfaces. One can translate the geometric properties of Calabi-Yau varieties into algebraic properties of the derived categories of coherent sheaves on them, and this leads to the notion of Calabi-Yau category that are now well studied. \emph{Fractionally Calabi-Yau} categories satisfies weaker axioms and there are more subtle geometric examples of such categories. They also appear in relation with singularity theory, mirror symmetry and Fukaya categories (See Section $6$ of \cite{chapoton_derived_tamari} for a conjecture relating Tamari lattices and \emph{Fukaya categories}). 
 
 In representation theory, they appear naturally as the bounded derived categories of the path algebra of an orientation of a Dynkin diagram or as derived categories of coherent sheaves on a  weighted projective line of tubular type (\cite{roosmalen}). Conjecture \ref{conj_A} provides an algebraic and combinatorial example with \emph{very different} homological properties: the incidence algebras of the Tamari lattices are not hereditary in general, and when $n\geqslant 4$, they have a wild representation type (\cite{chapoton_rognerud_wildness}).

 Our main objective in this article is to prove Conjecture \ref{conj_A}. The first part of the article is mostly algebraic and it reduces the conjecture to a much easier problem only involving a finite family of objects in the derived category of the Tamari lattice. The second part of the article is mostly combinatorial and is devoted to the understanding of this finite family.

In Section \ref{representation}, we define the category of $\kk$-linear representations of an arbitrary finite poset as the category of finitely generated right-modules over its incidence algebra. We recall that this category has a finite global dimension (bounded by the length of the largest chain in the poset). This implies that its bounded derived category has a so-called \emph{Serre functor}. This is a universal endo-functor of the derived category which induces, up to a sign, the Coxeter transformation on the Grothendieck group. This leads to the definition of a fractionally Calabi-Yau category (Definition \ref{def_CY}). A bounded derived category is \emph{fractionally Calabi-Yau} if some power of the Serre functor is isomorphic to a shift. 

The classical algebraic example is the bounded derived category of the path algebra of an orientation of a Dynkin diagram. Here we illustrate it in Example \ref{ex_cy} with the slightly more complicated case of an algebra derived equivalent to an orientation of $D_5$. This example is of course trivial from an algebraic point of view, but it illuminates the difficulty of understanding the images of indecomposable modules under the Serre functor.  
 
In Section \ref{poset_min}, we prove the first main result of the article which reduces the property of being fractionally Calabi-Yau to a property involving only finitely many objects of the derived category. 
 
 \begin{theorem}\label{theoA}
Let $X$ be a finite poset with a unique minimal element or a unique maximal element. Let $A_\kk(X)$ be the incidence algebra of $X$. If there are integers $m$ and $n$ such that $\Se(P)^n \cong P[m]$ in $D^b(X)$ for all projective indecomposable $A_\kk(X)$-modules, then the category $D^b\big(A_\kk(X)\big)$ is $(m,n)$-fractionally Calabi-Yau.
\end{theorem}

In the rest of this article we check the hypothesis of Theorem \ref{theoA} for the Tamari lattices. As it is suggested by Example \ref{ex_cy}, this is the difficult part of the proof. In general for a finite poset, the orbits of the projective indecomposable modules under the action of the Serre functor contains indecomposable complexes (not isomorphic to a shift of a module). However, for the Tamari lattices, it turns out that these orbits only contain shifts of indecomposable modules. We actually prove this on a larger family of modules that we call \emph{exceptional}. We show that the Serre functor has a nice combinatorial description on this family. 

 Exceptional modules come from Chapoton's work on the Dendriform operad $\mathbf{Dend}$. More precisely, in \cite{chapoton_mould} he introduced the operad $\mathbf{NCP}$ of noncrossing plants and the operad $\mathbf{Mould}$ of functions with a `variable number of variables'. Then, he proved that there is a commutative diagram of operads

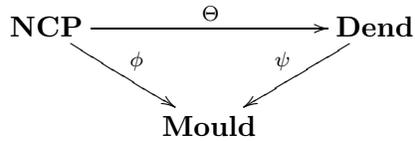
\[
\xymatrix{
\mathbf{NCP}\ar[rd]^{\phi}\ar[rr]^{\Theta} & & \mathbf{Dend} \ar[dl]_{\psi} \\
& \mathbf{Mould}
}
\]
where all the morphisms are injective. 

The family of noncrossing plants contains the more classical family of noncrossing trees (See Section \ref{bij_ip} for a precise definition). Since the elements of the Dendriform operad are linear combinations of binary trees, we can identify them with the elements of the Tamari lattices. It was proved in \cite{chapoton_etal_mould} that the image of a noncrossing tree by $\Theta$ is of the form $\sum_{t\in I} t$ where $I$ is an interval in a Tamari lattice. Such an interval is called \emph{exceptional}. For our purpose, an interval in the Tamari lattice is seen as an indecomposable module over its incidence algebra. As a consequence, in the derived category of the incidence algebra of the Tamari lattice there is a family of indecomposable objects in bijection with the noncrossing trees.

Since not all the intervals of the Tamari lattices are exceptional, the first task was to characterize the exceptional intervals. This was done in a previous article (\cite{rognerud_interval_poset}). They have a simple characterization in terms of the so-called \emph{interval-posets} of the Tamari lattices recently introduced by Ch\^atel and Pons (\cite{pons_chatel}). This characterization is recalled in Section \ref{section_tam}, where we also introduce a new bijection between noncrossing trees and exceptional interval-posets which may be of independent interest. 

In Section \ref{section_dico}, we show that the family of exceptional intervals contains the simple modules, the projective indecomposable modules and the injective indecomposable modules. We also show that they have particularly nice `boolean' projective resolutions. In Section \ref{se_fun}, we use these projective resolutions in order to obtain a first description of the action of the Serre functor on these objects in terms of the two bijections introduced in Section \ref{bij_ip}. Finally, in Section \ref{duality}, we investigate a \emph{planar duality} of noncrossing trees and we show that the Serre functor acts as a planar dual up to a shift on this family. In conclusion, we prove the following result. 

\begin{theorem}\label{theoB}
Let $T$ be a noncrossing tree of size $n$. Let $I_T$ be the corresponding indecomposable object in $D^b(\Tam_n)$. Then, there is an integer $n_T$ such that 
\[\Se(I_T) \cong I_{T^*}[n_T] \]
where $I_{T^*}$ is the indecomposable object in $D^b(\Tam_n)$ corresponding to the planar dual of $T$.
\end{theorem}

Since the square of the planar duality is nothing but a \emph{rotation}, we see that up to a shift the Serre functor has a finite order on these objects. The proof of Conjecture \ref{conj_A} follows from Theorems \ref{theoA} and \ref{theoB} and a combinatorial interpretation of the integers $n_T$ in terms of certain edges in the noncrossing trees $T^*$.

The image of the Serre functor on the Grothendieck group of $D^b(\Tam_n)$ induces up to shift the so-called Coxeter transformation of the incidence algebra of $\Tam_n$. As a consequence, we recover the result of Chapoton saying that the Coxeter transformation is periodic of period $2n+2$. Our approach is also interesting at this level since it gives a generating set of the Grothendieck group on which we clearly see why the Coxeter transformation is periodic. In particular, we are able to deduce new informations about the Coxeter matrix (See Proposition \ref{cox_mat_new}).

\paragraph{Acknowledgement}
This work was started when I was a postdoc at the University of Strasbourg and I am grateful to Frédéric Chapoton for introducing me to this subject, for his support, his comments and the many things he taught me. In particular for explaining to me the relevance of the noncrossing trees for Conjecture \ref{conj_A} and for sharing with me his conjectural version of Theorem~\ref{theoB}.

 \section{Representations of finite posets}\label{representation}
In the literature, there are two main points of view on the representations of finite posets. Here we consider representations of the incidence algebra of the poset. This is very different from the other kind of ``representations of posets", that was considered by Nazarova, Kleiner and others \cite{kleiner,nazarova,simson} which involves a non-abelian sub-category of the category of modules over a one-point extension of the Hasse diagram.
 
The \emph{incidence algebra} $A_{\kk}(P)$ of a poset $P$ over a field $\kk$ has a basis
$(a,b)$ indexed by pairs of comparable elements $a \leq b$ in $P$
and its associative product is defined on this basis by
\begin{equation*}
  (a,b) (c,d) = 
  \begin{cases}
    0 &\text{if}\quad b \not= c,\\
    (a,d) &\text{if}\quad b = c.
  \end{cases}
\end{equation*}

As a consequence, $A_{\kk}(P)$ is a finite dimensional algebra whose dimension is the number of intervals in the poset $P$. In order to avoid confusion, we reserve the notation $(a,b)$ for the basis elements of the incidence algebra. The intervals of the poset $P$ will be denoted by $[a,b]$.

In this article, we define the category of $\kk$-linear representations of the finite poset $P$ as the category of finite dimensional \emph{right-modules} over the incidence algebra $A_\kk (P)$. 
\begin{remark}
Alternatively, we can define the category of representations of a finite poset as the category of functors from the poset to the category of vector spaces. It is classical that this category is equivalent to the category of right modules over the incidence algebra. We think that both points of view are useful, but in order to simplify the exposition we restrict ourself to the category of modules over the incidence algebra. 
\end{remark}

The category $\Rpk$ is an \emph{abelian} category with \emph{enough projective} objects and \emph{enough injective} objects. 

For $x\in P$, let $P_x = (x,x)\cdot A_\kk(P)$. Similarly we let $I_x = \Hom_\kk(A_\kk(P)\cdot (x,x),\kk)$. Let $S_x$ be the $A_\kk(P)$-module of dimension $1$ on which every basis element acts by $0$ except for $(x,x)$ which acts by $1$. 

\begin{proposition}\label{basic_module}
Let $P$ be a finite poset and $\kk$ be a field. Let $x,y\in P$.  
\begin{enumerate}
\item The module $S_x$ is simple.
\item The module $P_x$ is projective and the module $I_x$ is injective. They are both indecomposable. $P_x$ is a projective cover of $S_x$ and $I_x$ is an injective envelope of $S_x$. 
\item \[ \Hom_{A_\kk(X)}\big(P_x, P_y\big) \cong \Hom_{A_\kk(X)}\big(I_x, I_y\big)  \cong \left\{\begin{array}{c}\kk \hbox{ if $y \leqslant x$, } \\0 \hbox{ otherwise.}\end{array}\right.\]
\item A nonzero morphism between two projective indecomposable modules is a monomorphism.
\item A nonzero morphism between two injective indecomposable modules is an epimorphism. 
\item $\Rpk$ has a finite global dimension bounded by the largest chain in $P$.
\end{enumerate}
\end{proposition}
\begin{proof}
For simplicity, we denote $A_\kk(X)$ by $A$. The module $S_x$ is one-dimensional, so it is simple. Since $(x,x)$ is an idempotent, the module $P_x$ is projective. The module $I_x$ is the dual of a projective left $A$-module, so it is injective. Since the idempotent $(x,x)$ is clearly primitive, both modules are indecomposable. 

We have $\Hom_{A}\big(P_x, P_y\big)\cong (y,y)A (x,x)$. It is a one-dimensional vector space with basis $(y,x)$ when $y\leqslant x$ and it is zero otherwise. The left multiplication by $(y,x)$ induces an injective morphism from $P_x$ to $P_y$. Every morphism between $P_x$ and $P_y$ is a scalar multiple of this injective morphism, so it is either zero or injective. The proof for the injective modules is dual. 

It is easy to see that the radical of $P_x$ has basis the set of $(x,z)$ such that $x\leqslant z$ and $z\neq x$. In particular $P_x$ surjects onto $S_x$ so it is a projective cover. The first term of a minimal projective resolution of $S_x$ is $P_x$. The composition factors of the kernel are given by the $S_y$ for $x\leqslant y$ and $x\neq y$. Then, either the kernel is projective or it has a projective cover of the form $\bigoplus_{z'}P_{z}$ where $x\leqslant z'$ is a cover relation. By induction, we see that $S_x$ has a finite minimal projective resolution. Moreover, the global dimension of the category is bounded by the largest chain in $P$. 
\end{proof}
\begin{remark}
In the proof above, we see that the global dimension of $P$ is bounded by the largest chain in $P$. However, it is not true that the global dimension has an easy combinatorial description. There are finite posets $P$ such that the global dimension of their incidence algebra depends on the characteristic of the ground field (See Proposition $2.3$ of \cite{igusa} for more details). 
\end{remark}

Since the category $\Rpk$ is abelian, one can consider its bounded derived category. Note that the ground field does not play an important role in our work, so by a slight abuse of notation we denote by $D^b(P)$ the bounded derived category of the category $A_\kk(P)\modu$. If the category $\Rpk$ is abelian, the bounded derived category $D^b(P)$ is a more subtle object. It has a structure of \emph{triangulated category} (for more details about triangulated categories and derived category we refer to Chapters $3$ and $6$ of \cite{zimmermann}). 

Informally, the objects of the derived category are left and right bounded chain complexes of $A_\kk(P)$-modules (See below in the next Paragraph). The morphisms are more complicated, but by construction if two complexes are quasi-isomorphic, they become isomorphic in the derived category. There is a fully faithful embedding of the category $\Rpk$ into the derived category $D^b(P)$. It is obtained by sending a right $A_\kk(P)$-module to the complex consisting of this module in degree zero. 

In the rest of this article we will \emph{always} work in the derived category, so when we speak about the $A_\kk(P)$-module $M$, we have in mind the image of this module in the derived category by the fully faithful functor. Since any $A_\kk(P)$-module is quasi-isomorphic to a projective resolution, in the derived category, \emph{we can identify a module with one of its projective resolutions}. 

The derived category has a \emph{suspension functor}. Here we call it the \emph{shift} and denote it by $[1]$. Let $(C_i,d_i)$ be a chain complex of $A_\kk(P)$-modules. That is $C_i$ is a $A_\kk(P)$-module and $d_i : C_i\to C_{i-1}$ is a morphism of $A_\kk(P)$-modules such that $d_{i-1}\circ d_i = 0$ for every $i$. Then $[1]$ denotes the usual left shift. More generally, for $n\in \mathbb{Z}$, we let $(C[n])_k = C_{k-n}$. 

A Serre functor of the derived category $D^b(P)$ is a self triangulated equivalence $\mathbb{S}$ of $D^b(P)$ such that there is a bi-functorial ismorphism 
\[
\Hom(X,\Se(Y)) \cong \Hom(Y,X)^*, \forall X,Y\in D^b(P),
\]
where $-^*$ denotes the $\kk$-linear dual. Since the global dimension of the incidence algebra $A_{\kk}(P)$ is finite, it is classical that the derived Nakayama functor $-\otimes_{A_\kk(P)}^{L} A_\kk(P)^*$ gives a Serre functor for $D^{b}(A)$ (See Theorem $3.1$ \cite{keller} for example). Note that Serre functors are unique up to unique isomorphism, so sometimes we will abusively speak about `the' Serre functor. 

Let $x\in P$. It is easy to see that $\Se(P_x) \cong I_x$. Let $y\leqslant x$. Let $i : P_x \to P_y$ be the embedding. Then $\Se(i) : I_x \to I_y$ is the canonical epimorphism. So, in the facts, when we want to compute the value of $\Se$ on a module $M$, we start by finding a projective resolution $\mathcal{P}_M$ of $M$. Then, we replace every projective indecomposable module by the corresponding injective module and the embeddings are replaced by the epimorphisms.  

\begin{example}\label{ex_dyck}
Even in the very simple case of representation of finite poset, the image of a module by a Serre functor is not always a module. For example, let $P$ be the poset $\{0,1,2,3,4\}$ where the cover relations are $0<1$, $0<2$, $1<3$, $2<3$ and $3<4$. In other words, it is the poset with Hasse diagram: 
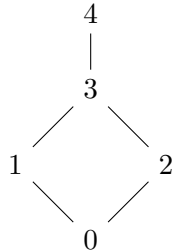
\[
\begin{tikzpicture}[scale=1]
\node (a) at (0,3) {$4$};
\node (b) at (0,2) {$3$};
\node (c) at (-1,1) {$1$};
\node (d) at (1,1) {$2$};
\node (e) at (0,0) {$0$};
\draw (e)--(d);
\draw (d)-- (b) -- (a);
\draw (e)--(c) -- (b);
 \end{tikzpicture}
\]
It is easy to see that $\Se(P_3)\cong I_3$. In the derived category $I_3\cong P_4\hookrightarrow P_0$. So $\Se^2(P_3)\cong I_4\to I_0 \cong \operatorname{Rad}(P_0)[1]$. In the derived category we have $\operatorname{Rad}(P_0)\cong P_3 \overset{f}{\to} P_1\oplus P_2$ where $f$ is the diagonal embedding of $P_3$ in $P_1\oplus P_2$. So $\Se^3(P_3)\cong (I_3 \overset{g}{\to}I_1\oplus I_2)[1]$ where $g$ is the diagonal projection of $I_3$ onto $I_1\oplus I_2$. The kernel of $g$ is isomorphic to $S_3$ and the cokernel is isomorphic to $S_0$. So $\Se^3(P_3)$ has homology concentrated in two degrees. Since it is the image of an indecomposable object by an equivalence, it is an indecomposable object. This implies that $\Se^3(P_3)$ is an indecomposable object of $D^b(P)$ which is not isomorphic to a shift of a module. 
\end{example}
\begin{definition}\label{def_CY}
Let $A$ be a finite dimensional $\kk$-algebra with finite global dimension. Then $D^b(A)$ is $(m,n)$-fractionally Calabi-Yau if there is an isomorphism of functors $\Se^n\cong [m]$. 
\end{definition}
This means that for every object $X$ in $D^b(A)$ there is an isomorphism $\phi_X : \Se^n(X)\to X[m]$ and for every morphism $f : X\to Y$ in $D^b(A)$, the  following diagram is commutative
\[
\xymatrix{
\Se^n(X)\ar[r]^{\phi_X}\ar[d]^{\Se^n(f)} & X[m]\ar[d]^{f[m]}\\
\Se^n(Y)\ar[r]^{\phi_Y}& Y[m]. 
}
\]
\begin{example}\label{ex_cy}
Let $Q$ be a finite connected acyclic quiver. It is well-known that $\kk Q$ is Fractionally Calabi-Yau if and only if $Q$ is an orientation of a Dynkin diagram of type $A,D,E$ (see for example \cite{roosmalen}).
\end{example}

The particular case of incidence algebras of finite posets has recently received attention in the work of several authors (See for example \cite{chapoton_derived_tamari,webb,lenzing}). If $P$ is a finite poset, then it is easy to see that, up to a sign, the linear map induced by the Serre functor on the Grothendieck group of $D^b(P)$ is the so-called \emph{Coxeter} transformation. Its matrix in the basis of the simple modules is called the Coxeter matrix. It has a simple description as $C=I (I^{-1})^t$ where $I$ is the incidence matrix of $P$. There is also an abundance of studies on Coxeter transformations in the literature (See for example \cite{chapoton_coxeter_tamari,ladkani_coxeter,yildirim}).

In general, it is not easy to check that the derived category is fractionally Calabi-Yau, or that the Coxeter matrix has finite order because the classical families of modules (or classical basis of the Grothendieck groups) are not preserved by the Serre functor (see Example \ref{ex_dyck}).

Looking one last time at Example \ref{ex_dyck}, we see that even in a simple case, understanding `by brute force' the Serre functor is not so easy. Nevertheless, it is easy to check that $D^b(P)$ is triangulated equivalent to $D^b(\kk D_5)$ and it is well-known that this category is $(6,8)$-fractionally Calabi-Yau. In other terms, finding a `good' description of our triangulated category as bounded derived category of a well chosen finite dimensional algebra, or abelian category, can simplify a lot the understanding of the Serre functors. 

The situation is particularly well understood when one can find a hereditary such category: for a finite dimensional path algebra, or more generally for an abelian hereditary category, we know precisely when their bounded derived category is fractionally Calabi-Yau (See \cite{roosmalen,lenzing_hereditary}). However, this cannot be used in order to solve Conjecture \ref{conj_A}.

\begin{proposition}
If $n\geqslant 4$, then $D^b(\Tam_n)$ is not triangulated equivalent to $D^b(\mathcal{H})$ for an abelian hereditary category $\mathcal{H}$. 
\end{proposition}
\begin{proof}
If there is a hereditary abelian category $\mathcal{H}$ such that $D^b(\Tam_n)\cong D^b(\mathcal{H})$, then the representation type of $A_\kk(\Tam_n)$ is dominated\footnote{Here dominated means smaller in the order where (finite) $\leq$ (tame) $\leq$ (wild).} by the representation type of $\mathcal{H}$ (see Section $2$ of \cite{HZ} or Proposition $2.3$ of \cite{chapoton_rognerud_wildness}). Looking at the classification, we see that if $D^b(\mathcal{H})$ is fractionally Calabi-Yau derived category, then $\mathcal{H}$ has a finite or a tame representation type. But, when $n\geqslant 4$, the incidence algebras of the Tamari lattices have a wild representation type (Section $3$ of  \cite{chapoton_rognerud_wildness}).
\end{proof}

 \section{Posets with a unique maximal or a unique minimal element}\label{poset_min}
 We consider a finite poset $X$ and a field $\kk$. The aim of this section is to prove the following result. 
\begin{theorem}\label{reduction}
Let $X$ be a finite poset with a unique minimal element or a unique maximal element. If there are integers $m$ and $n$ such that $\Se(P)^n \cong P[m]$ in $D^b(X)$ for all projective indecomposable modules of $X$, then the category $D^b(X)$ is $(m,n)$-fractionally Calabi-Yau.
\end{theorem}

For simplicity, in all this section we denote by $A$ the incidence algebra of $X$ over the field $\kk$. 

Let $\phi$ be a $\kk$-algebra endomorphism of $A$. Then, one can twist the regular left action of $A$ on itself by $\phi$ : $a\cdot x := \phi(a)x$ for $a,x\in A$. We denote by $\,_{\phi} A_1$ the $A$-$A$-bimodule $A$ where the left action is twisted by $\phi$ and the right action is given by the right multiplication.

\begin{proposition}\label{iyama}
Let $X$ be a finite poset and $A$ its incidence algebra over a field $\kk$. Then, the following are equivalent:
\begin{enumerate}
\item There are integers $m$ and $n$ such that $\Se^m(P_x) \cong P_x[n]$ in $D^b(A)$ for all $x\in X$
\item There is a $\kk$-algebra automorphism $\phi$ of $A$ such that $\phi(e_x)=e_x$ for all $x\in X$ and such that 
\[\Se^m \cong -\otimes_{A}^{L}\,_{\phi}A_1 \circ [n].\]
\end{enumerate}
\end{proposition}
\begin{proof}
Let $e$ be an idempotent in $A$. The $A$-module morphism $\theta : eA \otimes_A \,_{\phi} A \to \phi(e)A$ defined by $\theta(\sum_{i=0}^{n}e a_i\otimes b_i) = \sum_{i=0}^{n}\phi(ea_i)b_i$ is an isomorphism of right $A$-modules with inverse the morphism $\psi$ that sends $\phi(e)b$ to $e\otimes b$. As consequence, we see that $2$ implies $1$. 

Conversely, we start by using Lemmas $4.1$, $4.2$ and Proposition $4.3$ of \cite{iyama}. For the convenience of the reader we recall the main arguments.

By additivity of the derived Nakayama functor, hypothesis $1$ implies that
\[
\Se^m (A) \cong A[n] \hbox{ in $D^b(A)$}.
\]

In other terms, the complex of $A$-$A$-bimodules $D(A)^{\otimes^{L} m} := \overbrace{D(A)\otimes^{L}\cdots \otimes^{L}D(A)}^{m}$ is isomorphic in $D^b(A)$ to $A[n]$. Let $M = D(A)^{\otimes^{L} m}[-n]$. Then $M$ is isomorphic to $A$ in $D^b(A)$. It means that $M$ has its homology concentrated in degree $0$ and that the $A$-$A$-bimodule $H^{0}(M)$ is isomorphic to $A$ as a right-module. It is classical that the second condition is equivalent to the existence of a $\kk$-algebra endomorphism $\phi$ such that $H^{0}(M)\cong \,_{\phi}A_1$ as $A$-$A$-bimodules (See e.g. Lemma $4.1$ of \cite{iyama}). In other terms, the complex of bimodules $M$ has its homology concentrated in degree $0$ and its homology is isomorphic to $\,_{\phi}A_1$. So, we have that $M \cong \,_{\phi}A_1$ in $D^b(A\otimes_{\kk} A^{op})$ (this is the statement of Lemma $4.2$ of \cite{iyama}). This isomorphism of complexes of bimodules induces an isomorphism of functors 
\[\Se^{m} \cong -\otimes_{A}^{L}\,_{\phi}A_1\circ [n].\]  
Since $\Se^m$ is an autofunctor, we see that $\phi$ is a $\kk$-algebra automorphism. In other terms, the $\kk$-algebra $A$ is a twisted fractionally Calabi-Yau algebra in the sense of Herschend and Iyama (this it the statement of Proposition $4.3$ of \cite{iyama}).

By hypothesis, if $x\in X$, we have $\Se^{m} (e_x A) \cong e_xA[n]$. Since $\,_{\phi}A\otimes_A e_x A\cong \phi(e_x)A[n]$, we see that $e_x A \cong \phi(e_x)A$ as right $A$-modules. In other words, the idempotents $e_x$ and $\phi(e_x)$ are conjugated. Since $A/J(A)$ is commutative, the idempotents $e_x$ and $\phi(e_x)$ have the same image in the quotient. It follows that $e_x -\phi(e_x)\in J(A)$.

The element $g=\sum_{x\in X}e_x\phi(e_x)$ is invertible since $1-g\in J(A)$. Moreover, $g\phi(e_x)g^{-1} =e_x$ for all $x\in X$. Let $\phi' = c_{g} \circ \phi $ where $c_{g}$ denotes the conjugation by $g$. Then, $\phi'$ is a $\kk$-algebra automorphism of $A$ such that $\phi'(e_x)=e_x$ for all $x\in X$. Finally, since $\,_{c_{g}}A_1 \cong \,_1 A_1$ as bimodules, there is an isomorphism of functors between $-\otimes_A^{L}\,_{\phi}A_1$ and $-\otimes_{A}^{L}\,_{\phi'}A_1$.

\end{proof}
\begin{lemma}[Stanley \cite{stanley}]\label{stanley}
Let $X$ be a finite poset with a unique minimal or a unique maximal element. Let $\phi$ be a $\kk$-algebra automorphism of $A$ such that $\phi(e_x)=e_x$ for all $x\in X$. Then, $\phi$ is an inner automorphism.
\end{lemma}
\begin{proof}
We only give a proof when $X$ has a unique minimal element, the other case is similar. 

We denote the minimal element of $X$ by $0$. The vector space $e_x A e_y$ is $\{0\}$ when $x\nleqslant y$ and is generated by $(x,y)$ when $x\leqslant y$. For every basis element $(x,y)$ there is a scalar $k_{x,y}\in \kk^{*}$ such that $\phi(x,y)= k_{x,y}(x,y)$. 

Let $u = e_0 + \sum_{x\neq 0}k_{0,x}e_x$. Then $u$ is an invertible element with inverse $u^{-1} = e_0 + \sum_{x\neq 0} k_{0,x}^{-1}e_x$. It is easy to see that $u (0,x) u^{-1} = k_{0,x}^{-1} (0,x)$. In other words, $\phi\circ c_u (0,x) = (0,x)$ and $\phi\circ c_u(e_x) = e_x$ for all $x\in X$. 

For $0\neq x < y$, we have $(0,y) = (0,x)\cdot (x,y)$. This implies that an automorphism that fixes all the basis elements $(0,x)$ and all the idempotents $e_x$ is the identity. 
\end{proof}
\begin{proof}[Proof of Theorem \ref{reduction}.]
By Proposition \ref{iyama}, under the hypothesis we have $\Se^{m} \cong - \otimes_{A}^{L}\,_{\phi}A_1$ where $\phi$ is a $\kk$-algebra automorphism that fixes the idempotents $e_x$. By Lemma \ref{stanley} any such automorphism is inner, so the induced bimodule $\,_{\phi}A_1$ is isomorphic to the regular bimodule. It shows that $\Se^{m} \cong [n]$. 
\end{proof}
 \section{Tamari lattices, interval-posets and noncrossing trees}\label{section_tam}
 
 \subsection{Tamari lattices and interval-posets}
Let $n\in \mathbb{N}$. A (planar) binary tree of size $n$ is a graph embedded in the plane which is a tree, has $n$ vertices with valence $3$, $n+ 2$ vertices with valence $1$ and a distinguished univalent vertex called the \emph{root}. The other vertices of valence $1$ are called the \emph{leaves} of the tree. For the rest of the paper, when we speak about vertices of the tree, we have in mind the trivalent vertices. The planar binary trees are pictured with their root at the bottom and their leaves at the top.

Let $\Tam_n$ be the set of all binary trees with $n$ vertices. It is well-known that the cardinality of this set is the Catalan number $c_n = \frac{ 1 }{n+1} {{2n}\choose{n}}$. 

There is a partial order relation on $\Tam_n$ which was introduced by Tamari in \cite{tamari}. It is defined as the transitive closure of the following covering relations. A tree $T$ is covered by a tree $S$ if they only differ in some neighborhood of an edge by replacing the configuration $\begin{tikzpicture}[scale = 0.1]
        \draw (0,0)--(2,2); \draw (0,0)--(-2,2); \draw (1,1)--(0,2);
\end{tikzpicture}$ in $T$ by the configuration  $\begin{tikzpicture}[scale = 0.1]
        \draw (0,0)--(2,2); \draw (0,0)--(-2,2); \draw (-1,1)--(0,2);
\end{tikzpicture}$ in $S$. The poset $\Tam_n$ is known to be a lattice (See e.g. \cite{Huang_Tamari,tamari}). 

Our convention is maybe not the most classical: in general we use the left rotation of trees in order to describe the partial order of the Tamari lattices. We make this choice since it yields to a more natural description of some bijections. 

A binary \emph{search tree} is a binary tree labeled by integers such that if a vertex $x$ is labeled by $k$, then the vertices of the left subtree (resp. right subtree) of $x$ are labeled by integers  less than or equal (resp. superior) to $k$.

If $T$ is a binary tree with $n$ vertices, there is a unique labeling of the vertices by each of the integers $1,2,\cdots, n$ that makes it a binary search tree. This procedure is sometimes called the \emph{in-order traversal} of the tree (visit left subtree, root and then right subtree).

Using this labeling, a binary tree $T$ with $n$ vertices induces a partial order relation $\lhd$ on the set $\{1,\cdots, n\}$ by setting $i \lhd j$ if and only if the vertex labeled by $i$ is in the subtree with root $j$. 

When $(P,\lhd)$ is a partial order on the set $\{1,\cdots, n\}$, one can use the natural total ordering of the integers $1,\cdots, n$ that we denote by $<$ to split the relations $\lhd$ in two families. Let $1 \leqslant a < b \leqslant n$ be two integers. If $a \lhd b$ we say that the relation is \emph{increasing}.  On the other hand, if $b\lhd a$, we say that the relation is \emph{decreasing}. We denote by $\mathrm{Dec}(P)$ and $\mathrm{Inc}(P)$ the set of decreasing and increasing relations of $P$.  
 
If $T$ is a binary tree, we will implicitly see it as a binary search tree. Then, we have a useful characterization due to Châtel, Pilaud and Pons \cite{chatel_pilaud_pons} of the partial order of the Tamari lattice in terms of increasing or decreasing relations.
\begin{proposition}\label{char_tamari}
Let $T_1$ and $T_2$ be two binary trees. Then $T_1 \leqslant T_2$ in the Tamari lattice if and only if $\mathrm{Dec}(T_2) \subseteq \mathrm{Dec}(T_1)$ if and only if $\mathrm{Inc}(T_1) \subseteq \mathrm{Inc}(T_2)$. 
\end{proposition}
\begin{proof}
See Proposition $40$ and Remark 52 of \cite{chatel_pilaud_pons}. Note that our convention differs from theirs. 
\end{proof}
Let us recall the definition of interval-posets.
\begin{definition}
An interval-poset $(P, \lhd) $ is a poset over the integers $1,\cdots, n$ such that 
\begin{enumerate}
\item If $a\lhd c$ and $a<c$, then for all integers $b$ such that $a<b<c$, we have $b\lhd c$.
\item If $c\lhd a$ and $a<c$, then for all integers $b$ such that $a<b<c$, we have $b\lhd a$. 
\end{enumerate}
\end{definition}
The conditions $(1)$ and $(2)$ of this definition will be referred as the \emph{interval-poset condition}. The integer $n$ in the definition is called the \emph{size} of the interval-poset. 

\begin{theorem}[Châtel, Pons]\label{bij}
Let $n \in \mathbb{N}$. There is a bijection between the set of intervals in $\Tam_n$ and the set of interval-posets of size $n$. 
\end{theorem}
\begin{proof}
This is Theorem $2.8$ of \cite{pons_chatel}.

Since we need to use the explicit version of the theorem, we recall the bijections. if $[S,T]$ is an interval in $\Tam_n$, we can construct an interval-poset as follows. The trees $S$ and $T$ can be seen as binary search trees and they induce two partial order relations $\lhd_S$ and $\lhd_T$. Let $P = \{1,2,\cdots, n\}$. Let $\lhd$ be the binary relation on $P$ given by the disjoint union of the increasing relations of $S$ and the decreasing relations of $T$.  Then, it is proved in \cite{pons_chatel} that $(P,\lhd)$ is an interval-poset. 

Conversely, if $(P,\lhd)$ is an interval-poset of size $n$. Let $D$ be the poset obtained from $P$ by keeping only the decreasing relations of $P$. Similarly let $I$ be the poset obtained by keeping the increasing relations. By Lemma $2.5$ of \cite{pons_chatel}, the Hasse diagrams of these two posets are two forests. If we add a common root to the trees of each of these forests, we obtain two planar trees. Now, using the so-called \emph{knuth} rotations, we produce two binary trees starting from these planar trees.

For $I$ we recursively produce a binary tree $S$ by using the rule: right brother becomes right son and son becomes left son.

For $D$ we recursively produce a binary tree $T$ by using the rule: left brother becomes left son and son becomes right son.

The tree $S$ is smaller than $T$ for the order of the Tamari lattice, so we have an interval $[S,T]$.

It was proved in Theorem $2.8$ of \cite{pons_chatel} that these two constructions give two bijections inverse of each other. 
\end{proof}
\subsection{Noncrossing trees and exceptional interval-posets}\label{bij_ip}

Using three operads, it was proved in \cite{chapoton_mould,chapoton_etal_mould}, that one can associate an interval in the Tamari lattice to any noncrossing tree. In \cite{rognerud_interval_poset} we obtained a description of these intervals in terms of interval-posets avoiding certain configurations that we called \emph{exceptional}. For the purpose of this article, we can forget about the operads behind this description.  

We start by recalling the definitions of noncrossing trees and exceptional interval-posets. Then, we recall the bijections obtained in \cite{rognerud_interval_poset} between the noncrossing trees and the exceptional interval-posets. Finally, we give a new bijection between these two families. Both of these bijections will be used in the rest of the article.

Let $\mathcal{C}_n$ be a circle with $n+1$ equidistant points labeled in the clockwise order from $1$ to $n+1$. When we draw $\mathcal{C}_n$, we will always put the vertex with label $1$ on the bottom. 

A \emph{noncrossing tree} in $C_n$ is a set of edges between the marked points with the following properties 
\begin{itemize}
\item edges do not cross pairwise,
\item any two points are connected by a sequence of edges,
\item There is no loop made of edges.
\end{itemize}
See Figure \ref{fig1} for an example. We say that $n$ is the size of the tree and we denote by $\NCT_n$ the set of all noncrossing trees of size $n$. It is classical that the number of noncrossing trees in $C_n$ is $\frac{1}{2n+1}{{3n}\choose{n}}$ (See e.g. Theorem $3.10$ of \cite{dulucq}).

Let $T$ be a noncrossing tree in $\mathcal{C}_n$. An edge between two vertices $x$ and $y$ is denoted by $[x-y]$. An edge between $i$ and $i+1$ in $T$ is called a \emph{side} of $T$. If there is no edge from $i$ to $i+1$, we say that $[i-(i+1)]$ is an \emph{open side} of $T$. The side $[1-(n+1)]$ is called the base of $T$.

\begin{definition}\label{exc}
An interval-poset whose Hasse diagram does not contain any configuration of the form $y\to z$ and $y\to x$ where $x<y<z$ is called an \emph{exceptional} interval-poset.
\end{definition} 

Let us recall a bijection between exceptional interval-posets and noncrossing trees. For more details we refer to Section $3$ of \cite{rognerud_interval_poset}. 

If $(P,\lhd)$ is an exceptional interval-poset over the integers $[1,n]$ we can construct a list of edges in $\mathcal{C}_n$ by: 

for an integer $v$ consider the poset $\{ x\in [1,n]\ ;\ x\lhd v\}$. This poset has a minimal element (for the usual order relation $<$) $v_1$ and a maximal element $v_2$. We associate to $v$ the edge $[v_1 - (v_2+1)]$. We denote by $\psi(P)$ the set of these edges in $\mathcal{C}_n$. See Figure \ref{fig1} for an illustration. 

Conversely, let $T$ be a noncrossing tree of size $n$. 

\begin{enumerate}
\item We label the edges of the noncrossing $T$:

if an edge is a boundary side $[i,i+1]$, then it is labeled by $i$. If not, the edge separates a unique open side $[i,i+1]$ from the basis and it is labeled by $i$.

\item We define a relation $\lhd_T$ on $\{1,2,\cdots,n\}$ by $i \lhd_{T} j$ if the edge $i$ is separated from the base by the edge $j$. We let $\psi^{-1}(T)=\big(\{1,2,\cdots,n\},\lhd_T \big)$.   
\end{enumerate}

\begin{theorem}\label{image_noncrossing}
Let $n \in \mathbb{N}$. The map $\psi$ induces a bijection between the set of exceptional interval-posets of size $n$ and the set of noncrossing trees of size $n$ with $\psi^{-1}$ as inverse.
\end{theorem} 
\begin{proof}
See Section $3$ of \cite{rognerud_interval_poset}.
\end{proof}

Let us introduce another bijection between these two families.

Let $(P,\lhd)$ be an exceptional interval-poset of size $n$. Then, we consider $\widehat{P} = P\sqcup \{n+1\}$. We extend the relation $\lhd$ to $\widehat{P}$ by saying that $x\lhd n+1$ for $x\in P$. The new cover relations of $\widehat{P}$ are of the form $x\lhd (n+1)$ for $x$ a maximal element of $P$. We can see that $\widehat{P}$ is still an exceptional interval-poset because a new forbidden configuration would be of the form $y\lhd n+1$ and $y\lhd x$ for $x<y<n+1$. If $y\lhd n+1$ is a cover relation, then $y$ is a maximal element of $P$. This contradicts the relation $y\lhd x$. 

From the Hasse diagram of $\widehat{P}$ we construct $\theta(P)$ a set of edges in the circle $\mathcal{C}_n$ by sending the cover relation $i\lhd j$ to the edge $[i,j]$.

Conversely, let $T$ be a noncrossing tree in $\mathcal{C}_n$. Since $T$ is a noncrossing tree, for each vertex $i$ there is a unique path made of edges towards the vertex $n+1$. In particular, we can orient the edges of $T$ towards the vertex $n+1$. Then, a \emph{descent} is an edge from $i$ to $j$ where $i>j$ and a \emph{rise} is an edge from $i$ to $j$ where $i<j$.

From $T$ we consider the set $H$ of relations $i\lhd j$ where $[i-j]$ is an edge oriented from $i$ to $j$. See Figure \ref{fig1} for an illustration. 

\begin{theorem}\label{big_theta}
\begin{enumerate}
\item Let $P$ be an exceptional interval-poset of size $n$. Then, $\theta(T)$ is a noncrossing tree of size $n$.
\item Let $P$ be a noncrossing tree. Then, $H$ is the Hasse diagram of an exceptional interval-poset denoted by $\theta^{-1}(P)$. 
\item $\theta$ induces a bijection between the set of exceptional interval-posets of size $n$ and noncrossing trees in $\mathcal{C}_n$ with inverse $\theta^{-1}$. 
\end{enumerate}
\end{theorem}
\begin{proof}
The proof is not difficult and is left to the reader. Alternatively, in Proposition \ref{psironddual} we show that $\theta$ is the composition of $\psi$ and the planar duality for noncrossing trees. This will imply that it is a bijection. 
\end{proof}

 \begin{figure}
 \centering
\begin{tikzpicture}
\equic[1 cm]{5}
\draw[thick,red] (N3) -- (N2);
\draw[thick,red] (N1) -- (N4) -- (N5);
\draw[thick,blue] (N1) -- (N3);
\end{tikzpicture}
$\ \ \ $
\begin{tikzpicture}
\node (0) at (-5.5,4) {$\overset{\theta}{\longleftrightarrow}$}; \node (1) at (-4,4) {$1$}; \node (2) at (-3,4) {$2$}; \node (3) at (-2,4) {$3$}; \node (4) at (-1,4) {$4$}; \node (5) at (0.5,4) {$\overset{\psi}{\longleftrightarrow}$};
\draw[red,->,below,thick] (2.south) to [out=-45,in=225] (3.south);
\draw[red,->,below,thick] (1.south) to [out=-45,in=225] (4.south);
\draw[blue,->,above,thick](3.north) to [out=135,in=45] (1.north);
\end{tikzpicture}
$\ \ \ $
\begin{tikzpicture}
\equic[1 cm]{5}
\draw[thick] (N4) -- (N2)--(N3);
\draw[thick] (N1) -- (N5);
\draw[thick] (N1) -- (N4);
\end{tikzpicture}
\caption{Illustration of the bijections $\theta$ and $\psi$ between interval-posets and noncrossing trees.}\label{fig1}
\end{figure}
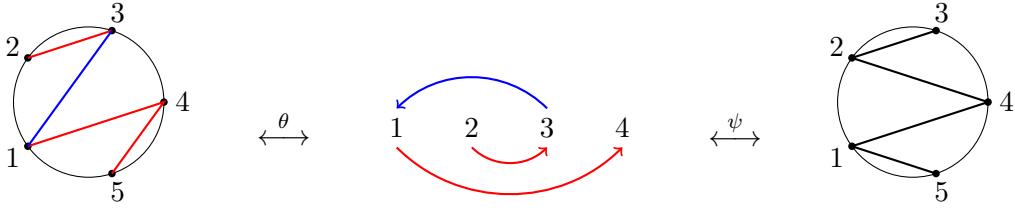

\section{Interval-posets in the derived category of the Tamari lattices}\label{section_dico}
 
 If $I$ is an interval in the Tamari lattice $\Tam_n$, then $\bigoplus_{x\in I}\kk e_x$ is an indecomposable $A_\kk(\Tam_n)$-module. For simplicity, we also denote this module by $I$. 

As explained in Theorem \ref{bij}, the intervals of the Tamari lattices have a nice description in terms of special partial order relations. So, starting with an interval-poset $P$ of size $n$, we implicitly use the bijection of Ch\^atel and Pons in order to have an interval $I_P$ of $\Tam_n$. Then, we construct the corresponding module, and finally we look at its image in the derived category. From now on, \emph{we will always identify the interval-posets with their corresponding indecomposable object in the derived category}. In this section we give a small dictionary that will allow us to work with them.

Let us start by a characterization of the interval-posets corresponding to simple modules.

\begin{proposition}\label{simple}
Let $I$ be an interval-poset of size $n$. Then, $I$ is a simple $A_{\kk}$-module if and only if for all $i < j $, there exists $i\leqslant z\leqslant j$ such that $i\lhd z$ and $j\lhd z$.
\end{proposition}
\begin{proof}
This is Proposition $39$ \cite{chatel_pilaud_pons}. 
\end{proof}
We also have a simple description of the projective indecomposable and the injective indecomposable. 
\begin{proposition}\label{dico}
Let $I$ be an interval-posets of size $n$.
\begin{enumerate}
\item $I$ is projective if and only if $\Dec(I)=\emptyset$. 
\item $I$ is injective if and only if $\Inc(I)=\emptyset$. 
\item Let $P_1$ and $P_2$ be two projective interval-posets. Then 
\[\Hom_{\RTk}(P_1,P_2) = \left\{\begin{array}{c}\kk \hbox{ if $\Inc(P_2)\subseteq \Inc(P_1)$,} \\0 \hbox{ otherwise.}\end{array}\right.\]
\item Let $I_1$ and $I_2$ be two injective interval-posets. Then 
\[\Hom_{\RTk}(I_1,I_2) = \left\{\begin{array}{c}\kk \hbox{ if $\Dec(P_1)\subseteq \Dec(P_2)$,} \\0 \hbox{ otherwise.}\end{array}\right.\]
\end{enumerate}
\end{proposition}
\begin{proof}
Let $I$ be an interval-poset and $i=[S,T]$ be the corresponding interval of the Tamari lattice. Then $i$ is projective if and only if $T$ is the maximal element of the Tamari lattice. The maximal element of $T$ is the left comb and it has no decreasing relation. Conversely, if $A$ is a binary tree, then the vertices of its right subtree are labeled by integers larger than the label of the root. So if $A$ has no decreasing relation, it is the left comb. 

By construction, the decreasing relations of $I$ are the decreasing relations of $T$. So $I$ is the interval-poset of a projective interval if and only if it has no decreasing relations. The proof for the injective interval-posets is dual. 

Let $x,y$ be two binary trees. By Proposition \ref{char_tamari}, we have $x\leqslant y$ if and only if $\Inc(x)\subseteq \Inc(y)$ if and only if $\Dec(y)\subseteq \Dec(x)$. The result follows from the description of the morphisms between projective modules and between injective modules given in Proposition \ref{basic_module}.

\end{proof}
As immediate corollary, we have:

\begin{corollary}
Let $n\in\mathbb{N}$. The projective indecomposable, the injective indecomposable and the simple $A_{\kk}(\Tam_n)$-modules are exceptional interval-posets. 
\end{corollary}

We will also need a technical characterization of the trees belonging in an interval.

\begin{lemma}\label{interval_char}
Let $I$ be an interval-poset of size $n$. Let $[S,T]$ be the corresponding interval of $\Tam_n$. 

Let $C$ (resp. $D$) be the set of increasing (resp. decreasing) relations of $I$ and $J$ be the set of decreasing relations that are in the Hasse diagram of $I$. We denote by $\overline{J}:=\{i\lhd j
 ;
  j\lhd i \in J\}$. 

Then, for a tree $x$, we have $x\in [S,T]$ if and only if $C\subseteq \Inc(x)$ and $\overline{J}\cap \Inc(x) = \emptyset$. 
\end{lemma}
\begin{proof}
By Proposition \ref{char_tamari}, $x\in [S,T]$ if and only if $C\subseteq \Inc(x)$ and $D\subseteq \Dec(x)$. Then, it is clear that a tree in $[S,T]$ satisfies $C\subseteq \Inc(x)$ and $\overline{J}\cap \Inc(x) = \emptyset$.

Let $j\lhd i$ be a decreasing relation in $I$ which cannot be written as a concatenation of two smaller relations (i.e. a cover relation in the poset of decreasing relations of $I$). Then, using the definition of interval-poset, it is not difficult to check that either $j\lhd i$ is a cover relation of $I$, or there exists $i<j<k$ such that $j\lhd i = j\lhd k \lhd i$ where $k\lhd i$ is a cover relation of $I$. 

Let $x$ be a tree such that $C\subseteq x$ and $\overline{J}\cap x =\emptyset$. We show by induction on the length of the relations that $D\subseteq I$.

Let $i+1\lhd i\in D$. By the discussion above, there is $i+1\leqslant k$ such that $i+1\lhd k \in C$ and $k\lhd i\in J$. By hypothesis $i+1\lhd k\in \Inc(x)$ and $i\lhd k\notin \Inc(x)$.

Since $x$ is simple, by Proposition \ref{simple} $i\lhd i+1 \in x$ or $i+1\lhd i \in x$. In the first case, by transitivity we obtain the relation $i\lhd k \in x$. This contradicts $x\cap \overline{J} = \emptyset$. 

Let $j\lhd i \in D$. By induction we can assume that $j\lhd i$ is a cover relation in the poset of decreasing relations. Then, there is $j\leqslant k$ such that $j\lhd k\in C\subseteq \Inc(x)$ and $k\lhd i\in J$. Since $x$ is simple, there is $i\leqslant z\leqslant j$ such that $i\lhd z\in x$ and $j\lhd z\in x$. 

If $i<z<j$, then by the interval-poset condition the relation $z\lhd i$ is in $I$. And by induction it is also in $x$. So there are two possibilities: either $z=i$ or $z=j$. The second possibility implies the existence of the relation $i\lhd k\in x$ which contradicts $x\cap \overline{J}=\emptyset$.

\end{proof}
As explained in Section \ref{representation}, the Nakayama functor sends the projective indecomposable module $P_T$ to the injective indecomposable module $I_T$. If $x$ is the interval-poset corresponding to a binary tree $T$, we obtain the interval-poset of $P_T$ by considering the increasing relations of $x$ and we obtain $I_T$ by considering the decreasing relations. 

In the facts, we have a projective interval-poset $P$ and we want to find the injective interval-poset corresponding to $\Se(P)$ without having to construct the tree $x$. In other terms, we need to understand how the decreasing relations of a simple interval-poset are characterized by its increasing relations. 

\begin{lemma}\label{rel_simple}
Let $x$ be a simple interval-poset. Let $C$ be the set of increasing relations of $x$. The set of decreasing relations of $x$ is given by

\[C^0 = \{j\lhd i\ ;\ \forall i<s \leqslant j,\ i\lhd s \notin C\}. \]

\end{lemma}
\begin{proof}
If $j \lhd i$ is a decreasing relation of $x$, then by the interval-poset condition for $i\leqslant s < j$ there is a relation $s\lhd i$ in $x$. This implies that $j\lhd i \in C^0$. 

Conversely, let $j\lhd i \in C^0$. Since $x$ is simple, there exists $i\leqslant t \leqslant j$ such that $i\lhd t$ and $j\lhd t$. The only possibility is to have $t=i$, so the decreasing relation $j\lhd i$ is in $x$.
\end{proof}
\begin{proposition}\label{serre_rong}
Let $P_C$ be a projective interval-poset with set of increasing relations $C$. Then, 
\[
\Se(P_C) \cong I_{C^0},
\]
where $I_{C^0}$ is the injective interval-poset with set of decreasing relations $C^0$. 
\end{proposition}
\begin{proof}
It is an immediate corollary of Lemma \ref{rel_simple}.
\end{proof}
We are now almost ready to produce projective resolutions for the intervals of the Tamari lattices in terms of interval-posets. 

Let $(I,\lhd)$ be an interval-poset of size $n$. Let $C = \Inc(I)$. Let $D=\Dec(I)$. And let $J$ be the set of decreasing relations that are in the Hasse diagram of $I$.

If $f = j \lhd i$ is a relation of $I$, we denote by $\overline{f}$ the relation $i\lhd j$. 

If $R\subseteq J$, we denote by $C+\overline{R}$ the smallest interval-poset containing the relations of $C$ and the relations $\overline{f}$ such that $f\in R$. In other words, $C+\overline{R}$ is the \emph{interval-poset closure} of $C\sqcup \overline{R}$. We denote by $P_{C+\overline{R}}$ the interval-poset consisting of the relations of $C+\overline{R}$. By Proposition \ref{dico} it is a projective interval-poset. 

If $R_1 \subset R_2 \subset J$, then $C\subset C+ \overline{R_1} \subset C+ \overline{R_2}$. By Proposition \ref{dico}, we have injective morphisms $P_{ C+ \overline{R_2}} \to P_{C+\overline{R_1}} \to P_C$.

Let $\Delta=\Delta^{|J|-1}$ be a standard $(|J|-1)-simplex$. We identify the vertices of $\Delta$ with the elements in $J$. Let $\Po(I)$ be the following complex of projective modules where the morphisms are given by the (signed) embedding of projective modules:
\[
P_{C+\overline{J}}\to \bigoplus_{\substack{R\subset J \\ |R|=|J|-1}} P_{C+ \overline{R}} \to \cdots \to \bigoplus_{\substack{R\subset J \\ |R|=1}} P_{C+\overline{R}} \to P_C,
\]
where $P_C$ is in degree $0$. It is well known that choosing an orientation of $\Delta$ allows us to assign a sign to each of the monomorphisms in such a manner that $\mathcal{P}(I)$ is a chain complex. In the derived category, all the different choices lead to isomorphic complexes.  

\begin{lemma}\label{reso_proj}
$\Po(I)$ is a projective resolution of the interval $I = [S ; T]$
\end{lemma}
\begin{proof}
It is easy to see that the homology is concentrated in degree zero. It is the largest quotient of $P_C$ which is supported by the set of trees $x$ such that $C \subseteq \Inc(x)$ and $\Inc(x) \cap \overline{J}=\emptyset$. So by Lemma \ref{interval_char}, the homology of the complex $\mathcal{P}(I)$ is isomorphic to the interval $[S,T]$. 

The surjective morphism from $P_C$ to $I$ induces a quasi-isomorphism between $\mathcal{P}(I)$ and $I$.
\end{proof}
\begin{remark}
In Lemma \ref{reso_proj} the interval $I$ is not assumed to be exceptional. Dually we can produce an injective resolution by adding to the set of decreasing relations of $I$ the opposite of the increasing covers of $I$.
\end{remark}

 \section{Serre functor on noncrossing trees}\label{se_fun}
 In this section we compute the image of an exceptional interval-poset by the Serre functor. The main ingredient is the existence of boolean projective resolutions obtained in Lemma \ref{reso_proj} and the description of the action of the Serre functor on a projective interval-poset obtained in Proposition \ref{serre_rong}. This description is obtained in terms of interval-posets and could be entirely proved in terms of interval-posets. However, the proof would be very technical and rather obscure. So, we use as much as possible the noncrossing trees and the bijections $\phi$ and $\theta$ described in Section \ref{bij_ip}. 
 \subsection{Descents and rises of noncrossing trees}
 Let $T$ be a noncrossing tree in the circle $\mathcal{C}_n$. The edges of $T$ are oriented towards the vertex $n+1$. 
 
 Let $f=[a-b]$ be an edge in $T$. Since there is no loop made of edges, and since any pair of vertices are connected by a succession of edges, there is a maximal vertex $i_f$ such that $a\leqslant i_f< b$ and $i_f$ is connected to $a$ in $T-\{[a-b]\}$. Similarly, there is a vertex $j_f$ maximal (in the cyclic ordering of the vertices) such that $b\leqslant j_f < a$ and $j_f$ is connected to $b$ in $T-\{[a-b]\}$.
 
 If the edge $f=[a-b]$ is a descent, it means that $a$ is after $b$ in the path that goes from $b$ to $n+1$. The path connecting $1$ to $n+1$ cannot go through the vertices $a$ or $b$. So, there must be an edge in $T$ between an element $1\leqslant a_1 \leqslant a$ and $b<b_1 \leqslant n+1$. If $a_1$ is minimal and $b_1$ is maximal for this property, the edge $[a_1,b_1]$ is the first (from the top to the bottom) edge that separates $[a-b]$ from the base. The situation is illustrated in the left case of Figure \ref{descente-montee}.
 
If the edge $f = [a-b]$ is a rise, the situation is slightly mode complicated. Since $a$ is before $b$ in the path from $a$ to $n+1$, we have $j_{f} = n+1$ or $1\leqslant j_{f}<a$. If $j_f = n+1$, then there is no edge that separates $f$ from the base. Otherwise, there is such an edge. The two situations are illustrated in Figure \ref{descente-montee}.
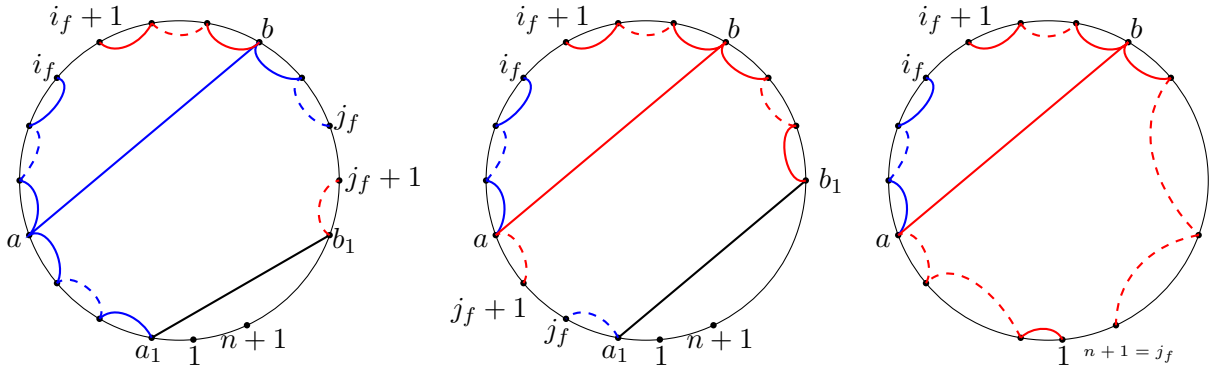
\begin{figure}[h]
\centering
\begin{tikzpicture}[scale=0.85]
\draw[thin] (0,0) circle (2.5cm);

\coordinate (Nx) at (-85:2.5);
\draw (-85:2.5+0.25) node{$1$};
\fill[black] (Nx) circle (0.05 cm);
\coordinate (Ny) at (-65:2.5);
\draw (-65:2.5+0.25) node{$n+1$};
\fill[black] (Ny) circle (0.05 cm);
\coordinate (N1) at (-100:2.5);
\draw (-100:2.5+0.25) node{$a_1$};
\fill[black] (N1) circle (0.05 cm);
\coordinate (N2) at (-120:2.5);
\draw (-120:2.5+0.25) node{};
\fill[black] (N2) circle (0.05 cm);
\coordinate (N3) at (-140:2.5);
\draw (-140:2.5+0.25) node{};
\fill[black] (N3) circle (0.05 cm);
\coordinate (N4) at (-160:2.5);
\draw (-160:2.5+0.25) node{$a$};
\fill[black] (N4) circle (0.05 cm);
\coordinate (N5) at (-180:2.5);
\draw (-180:2.5+0.25) node{};
\fill[black] (N5) circle (0.05 cm);
\coordinate (N6) at (-200:2.5);
\draw (-200:2.5+0.25) node{};
\fill[black] (N6) circle (0.05 cm);
\coordinate (N7) at (-220:2.5);
\draw (-220:2.5+0.25) node{$i_f$};
\fill[black] (N7) circle (0.05 cm);
\coordinate (N8) at (-240:2.5);
\draw (-240:2.5+0.4) node{$i_f+1$};
\fill[black] (N8) circle (0.05 cm);
\coordinate (N9) at (-260:2.5);
\draw (-260:2.5+0.25) node{};
\fill[black] (N9) circle (0.05 cm);
\coordinate (N10) at (-280:2.5);
\draw (-280:2.5+0.25) node{};
\fill[black] (N10) circle (0.05 cm);
\coordinate (N11) at (-300:2.5);
\draw (-300:2.5+0.25) node{$b$};
\fill[black] (N11) circle (0.05 cm);
\coordinate (N12) at (-320:2.5);
\draw (-320:2.5+0.25) node{};
\fill[black] (N12) circle (0.05 cm);
\coordinate (N13) at (-340:2.5);
\draw (-340:2.5+0.3) node{$j_f$};
\fill[black] (N13) circle (0.05 cm);
\coordinate (N14) at (-360:2.5);
\draw (-360:2.5+0.7) node{$j_f+1$};
\fill[black] (N14) circle (0.05 cm);
\coordinate (N15) at (-20:2.5);
\draw (-20:2.5+0.25) node{$b_1$};
\fill[black] (N15) circle (0.05 cm);

\draw[thick] (N15)--(N1);
\draw[thick,dashed,red] (N15) to [out=150,in=-170] (N14);
\draw[thick,dashed,blue] (N2) to [out=60,in=40] (N3)
                         (N5)to [out=45,in=-15](N6)
                         (N13) to [out=180,in=-160] (N12);
\draw[thick,blue] (N1)to [out=115,in=45](N2)
			 (N3)to [out=60,in=30](N4) to [out=50,in=-10](N5)
			 (N6) to [out=15,in=-15] (N7)
			 (N11)--(N4)
			 (N12) to [out=-160,in=-140] (N11);
\draw[thick,red] (N8) to [out=-45,in=-115] (N9)
                 (N10) to [out=-90,in=-130] (N11);
\draw[thick,red,dashed] (N9) to [out=-45,in=-130] (N10);
\end{tikzpicture}
\begin{tikzpicture}[scale=0.85]
\draw[thin] (0,0) circle (2.5cm);

\coordinate (Nx) at (-85:2.5);
\draw (-85:2.5+0.25) node{$1$};
\fill[black] (Nx) circle (0.05 cm);
\coordinate (Ny) at (-65:2.5);
\draw (-65:2.5+0.25) node{$n+1$};
\fill[black] (Ny) circle (0.05 cm);
\coordinate (N1) at (-100:2.5);
\draw (-100:2.5+0.25) node{$a_1$};
\fill[black] (N1) circle (0.05 cm);
\coordinate (N2) at (-120:2.5);
\draw (-120:2.5+0.25) node{$j_f$};
\fill[black] (N2) circle (0.05 cm);
\coordinate (N3) at (-140:2.5);
\draw (-140:2.5+0.7) node{$j_f+1$};
\fill[black] (N3) circle (0.05 cm);
\coordinate (N4) at (-160:2.5);
\draw (-160:2.5+0.25) node{$a$};
\fill[black] (N4) circle (0.05 cm);
\coordinate (N5) at (-180:2.5);
\draw (-180:2.5+0.25) node{};
\fill[black] (N5) circle (0.05 cm);
\coordinate (N6) at (-200:2.5);
\draw (-200:2.5+0.25) node{};
\fill[black] (N6) circle (0.05 cm);
\coordinate (N7) at (-220:2.5);
\draw (-220:2.5+0.25) node{$i_f$};
\fill[black] (N7) circle (0.05 cm);
\coordinate (N8) at (-240:2.5);
\draw (-240:2.5+0.4) node{$i_f+1$};
\fill[black] (N8) circle (0.05 cm);
\coordinate (N9) at (-260:2.5);
\draw (-260:2.5+0.25) node{};
\fill[black] (N9) circle (0.05 cm);
\coordinate (N10) at (-280:2.5);
\draw (-280:2.5+0.25) node{};
\fill[black] (N10) circle (0.05 cm);
\coordinate (N11) at (-300:2.5);
\draw (-300:2.5+0.25) node{$b$};
\fill[black] (N11) circle (0.05 cm);
\coordinate (N12) at (-320:2.5);
\draw (-320:2.5+0.25) node{};
\fill[black] (N12) circle (0.05 cm);
\coordinate (N13) at (-340:2.5);
\draw (-340:2.5+0.3) node{};
\fill[black] (N13) circle (0.05 cm);
\coordinate (N14) at (-360:2.5);
\draw (-360:2.5+0.4) node{$b_1$};
\fill[black] (N14) circle (0.05 cm);

\draw[thick] (N14)--(N1);
\draw[thick,dashed,blue] (N1)to [out=115,in=45](N2)
                         (N5)to [out=45,in=-15](N6);
\draw[thick,blue] 
			 (N4) to [out=50,in=-10](N5)
			 (N6) to [out=15,in=-15] (N7);
\draw[thick,red] (N8) to [out=-45,in=-115] (N9)
                 (N10) to [out=-90,in=-130] (N11)
                 (N11)--(N4)
                 
			 (N12) to [out=-160,in=-140] (N11)
			 (N13) to [out=180,in=-160] (N14);
\draw[thick,red,dashed] (N9) to [out=-45,in=-130] (N10)
                         (N13) to [out=180,in=-160] (N12)
                         (N3)to [out=60,in=30] (N4);
\end{tikzpicture}
\begin{tikzpicture}[scale=0.85]
\draw[thin] (0,0) circle (2.5cm);

\coordinate (Nx) at (-85:2.5);
\draw (-85:2.5+0.25) node{$1$};
\fill[black] (Nx) circle (0.05 cm);
\coordinate (Ny) at (-65:2.5);
\draw (-65:2.5+0.5) node{{\tiny $n+1=j_f$}};
\fill[black] (Ny) circle (0.05 cm);
\coordinate (N1) at (-100:2.5);
\draw (-100:2.5+0.25) node{};
\fill[black] (N1) circle (0.05 cm);
\coordinate (N3) at (-140:2.5);
\draw (-140:2.5+0.7) node{};
\fill[black] (N3) circle (0.05 cm);
\coordinate (N4) at (-160:2.5);
\draw (-160:2.5+0.25) node{$a$};
\fill[black] (N4) circle (0.05 cm);
\coordinate (N5) at (-180:2.5);
\draw (-180:2.5+0.25) node{};
\fill[black] (N5) circle (0.05 cm);
\coordinate (N6) at (-200:2.5);
\draw (-200:2.5+0.25) node{};
\fill[black] (N6) circle (0.05 cm);
\coordinate (N7) at (-220:2.5);
\draw (-220:2.5+0.25) node{$i_f$};
\fill[black] (N7) circle (0.05 cm);
\coordinate (N8) at (-240:2.5);
\draw (-240:2.5+0.4) node{$i_f+1$};
\fill[black] (N8) circle (0.05 cm);
\coordinate (N9) at (-260:2.5);
\draw (-260:2.5+0.25) node{};
\fill[black] (N9) circle (0.05 cm);
\coordinate (N10) at (-280:2.5);
\draw (-280:2.5+0.25) node{};
\fill[black] (N10) circle (0.05 cm);
\coordinate (N11) at (-300:2.5);
\draw (-300:2.5+0.25) node{$b$};
\fill[black] (N11) circle (0.05 cm);
\coordinate (N12) at (-320:2.5);
\draw (-320:2.5+0.25) node{};
\fill[black] (N12) circle (0.05 cm);
\coordinate (N15) at (-20:2.5);
\draw (-20:2.5+0.25) node{};
\fill[black] (N15) circle (0.05 cm);

\draw[thick,dashed,blue] 
                         (N5)to [out=45,in=-15](N6);
\draw[thick,blue] 
			 (N4) to [out=50,in=-10](N5)
			 (N6) to [out=15,in=-15] (N7);
\draw[thick,red] (Nx)to [out=115,in=45](N1)
                 (N8) to [out=-45,in=-115] (N9)
                 (N10) to [out=-90,in=-130] (N11)
                 (N11)--(N4)
			 (N12) to [out=-160,in=-140] (N11);
\draw[thick,red,dashed] (N1) to [out=115,in=45](N3)
						 (N12) to [out=-130,in=135] (N15)
                         (N3)to [out=60,in=30] (N4)
                         (N9)to [out=-60,in=-130] (N10)
                         (N15) to[out=-170,in=90] (Ny);
\end{tikzpicture}
\caption{Shape of the subtrees determined by the edge $[a-b]$. Edges in blue are descents, and edges in red are rises.}\label{descente-montee}
\end{figure}

\begin{lemma}\label{6.1}
Let $I$ be an exceptional interval-poset of size $n$. Let $x,y\in \{1,2,\cdots, n\}$.
\begin{enumerate}
\item There is an increasing cover relation $x\lhd y$ in $I$ if and only if $x$ is the label of a descent $f$ in $\psi(I)$. In this case $x=i_f$ and $y=j_f$. 
\item There is a decreasing cover relation $x\lhd y$ in $I$ if and only if $x$ is the label of a rise $f$ in $\psi(I)$ such that $j_f \neq n+1$. In this case $x = i_f$ and $y=j_f$.  
\item $x$ is a maximal element of $I$ if and only if it is the label of a rise $f$ in $\psi(I)$ such that $j_f = n+1$.
\end{enumerate}
\end{lemma}
\begin{proof}
If $x$ is the label of a descent $f = [a-b]$ in $\psi(I)$, looking at Figure \ref{descente-montee}, we see that $x=i_f$ and the first edge between $f$ and the base is labeled by $j_f$, so in $I$ we have $i_f\lhd j_f$. It is an increasing relation since $i_f < j_f$. 

Conversely, if $x\lhd y$ is an increasing cover relation in $I$, then there are two edges $f=[a,b]$ and $f_1 = [a_1,b_1]$ in $\psi(I)$ respectively labeled by $x$ and $y$ and such that $f_1$ is the first edge that separates $f$ from the base. This implies that the tree has the shape of the leftmost or of the middle cases of Figure \ref{descente-montee}. Since $x<y$, the open side is between $b$ and $b_1$. This implies that the path from $b$ to $n+1$ goes through $a$. In other terms $f=[a-b]$ is a descent. 

The proofs of the second and third points are similar. 
\end{proof}
 \subsection{Image of an exceptional interval-poset by the Serre functor}
For all this section $I$ denotes an \emph{exceptional} interval-poset of size $n$. We simply denote by $C$ (resp. $D$) its set of increasing (resp. decreasing) relations. We denote by $J$ (resp. $K$) the set of decreasing (resp. increasing) relations of $I$ that are in the Hasse diagram of $I$. 

We will use the bijection $\theta$ introduced in Section \ref{bij_ip} in order to associate to $I$ a noncrossing tree. 

The first step is to understand the image of the interval-poset $P_{C+\overline{J}}$ by the Serre functor. For this we need to understand $(C+\overline{J})^0$. This is easier after using the bijection $\theta$. We encourage the reader to use Figure \ref{descente-montee} in all the different proofs. Most of the arguments are rather tedious but come from graphical evidences. 

\begin{lemma}\label{char+delta}
Let $x\in \{1,2,\cdots ,n\}$.  
\begin{enumerate}
\item $x+1 \lhd x \in (C+\overline{J})^0$ if and only if for every $1\leqslant e \leqslant x$, the edge $[e-(x+1)]$ is not in $\theta(I)$.
\item In this case, there is a rise $f$ in $\theta(I)$ starting at $x+1$ such that $x=j_f$. We denote by $[a_1-b_1]$ the first edge (from top to bottom) that separates $f$ from the base.
\item There is a relation $y\lhd x \in (C+\overline{J})^{0}$ if and only if there is a rise $f=[(x+1)-b]$ in $\theta(I)$ with $j_f = x$ and $x+1\leqslant y \leqslant b_1 -1$ where $[a_1-b_1]$ is the first edge that separates $f$ from the base. 
\end{enumerate} 
\end{lemma}
 \begin{proof}
 By definition $x+1\lhd x\in (C+\overline{J})^0$ if and only if $x\lhd x+1$ is not in $C+\overline{J}$. The relation $x\lhd x+1$ is in $C+\overline{J}$ if and only if $x\lhd x+1 \in C$ or there is an element $1\leq e \leq x$ such that $x+1\lhd e \in J$. The second case is equivalent to the existence of a descent $[e-(x+1)]$ in $\theta(I)$. In the first case, the relation may not be a cover relation in $I$. However it is easy to see that there is $1\leq e \leq x$ such that $x\lhd x+1 = x\lhd e \lhd x+1$ and $e\lhd x+1$ is a cover relation of $I$. In other words, $x\lhd x+1 \in C+\overline{J}$ if and only if there is an edge $[e-(x+1)]$ in $\theta(I)$ for $1\leq e \leq x$.

 Since $x+1$ is connected to $n+1$ and there is no edge from $x+1$ to the elements of $\{1,2,\cdots ,x\}$, there is a unique rise $[(x+1)-b]$ starting at $x+1$. Similarly, there is an edge between an element in $\{1,2,\cdots, x\}$ and an element in $\{b,b+1,\cdots, n+1\}$. It separates $[(x+1)-b]$ from the basis. We choose this edge to be minimal (from top to bottom) for this property. Graphically, the noncrossing tree has the shape of Figure \ref{cj0}.

If $y\lhd x\in (C+\overline{J})^0$, then by the interval-poset condition we see that $(x+1)\lhd x \in (C+\overline{J})^0$. So by the first part of the proof there is a rise $[x+1,b]$ with $j_f = x$. 

The existence of the edge $[a_1-b_1]$ implies that $b_1\lhd x\notin (C+\overline{J})^0$. If $(b_1-1)\lhd x\in (C+\overline{J})$, then there is $1\leq e\leq x$ and $x < s \leq b_1 -1$ such that $[e-s] \in \theta(I)$. Since $x+1 \leq s \leq b_1 - 1$, looking at Figure \ref{cj0}, we see that the noncrossing condition implies that $a_1 \leq e$ and $b\leq s$. Then, the minimality of $[a_1-b_1]$ contradicts the existence of the edge $[e-s]$. 

 Since this set $(C+\overline{J})^0$ is stable by the interval-poset condition, the result follows.

%
%
%
%

 \end{proof}
 \begin{figure}[h]
 \centering
 \begin{tikzpicture}[scale=0.85]
 \newcommand\rad{2.5}
\draw[thin] (0,0) circle (2.5cm);

\coordinate (Nx) at (-85:2.5);
\draw (-85:2.5+0.25) node{{\small $1$}};
\fill[black] (Nx) circle (0.05 cm);
\coordinate (Ny) at (-65:2.5);
\draw (-65:2.5+0.25) node{{\small $n+1$}};
\fill[black] (Ny) circle (0.05 cm);
\coordinate (N1) at (-100:2.5);
\draw (-100:2.5+0.25) node{$a_1$};
\fill[black] (N1) circle (0.05 cm);
\coordinate (N3) at (-140:2.5);
\draw (-140:2.5+0.8) node{{\small $x=j_f$}};
\fill[yellow!95!black] (N3) circle (0.08 cm);
\coordinate (N4) at (-160:2.5);
\draw (-160:2.5+0.6) node{{\small $x+1$}};
\fill[yellow!95!black] (N4) circle (0.05 cm);
\coordinate (N5) at (-180:2.5);
\draw (-180:2.5+0.25) node{};
\fill[yellow!95!black] (N5) circle (0.05 cm);
\coordinate (N6) at (-200:2.5);
\draw (-200:2.5+0.25) node{};
\fill[yellow!95!black] (N6) circle (0.05 cm);
\coordinate (N7) at (-220:2.5);
\draw (-220:2.5+0.25) node{$i_f$};
\fill[yellow!95!black] (N7) circle (0.05 cm);
\coordinate (N8) at (-240:2.5);
\draw (-240:2.5+0.4) node{$i_f+1$};
\fill[yellow!95!black] (N8) circle (0.05 cm);
\coordinate (N9) at (-260:2.5);
\draw (-260:2.5+0.25) node{};
\fill[yellow!95!black] (N9) circle (0.05 cm);
\coordinate (N10) at (-280:2.5);
\draw (-280:2.5+0.25) node{};
\fill[yellow!95!black] (N10) circle (0.05 cm);
\coordinate (N11) at (-300:2.5);
\draw (-300:2.5+0.25) node{$b$};
\fill[yellow!95!black] (N11) circle (0.05 cm);
\coordinate (N12) at (-320:2.5);
\draw (-320:2.5+0.25) node{};
\fill[yellow!95!black] (N12) circle (0.05 cm);
\coordinate (N13) at (-340:2.5);
\draw (-340:2.5+0.7) node{$b_1 -1$};
\fill[yellow!95!black] (N13) circle (0.08 cm);
\coordinate (N14) at (-360:2.5);
\draw (-360:2.5+0.4) node{$b_1$};
\fill[black] (N14) circle (0.05 cm);

\draw[thick] (N14)--(N1);
\draw[thick,dashed,blue] (N1)to [out=115,in=45](N3)
                         (N5)to [out=45,in=-15](N6);
\draw[thick,blue] 
			 (N4) to [out=50,in=-10](N5)
			 (N6) to [out=15,in=-15] (N7);
\draw[thick,red] (N8) to [out=-45,in=-115] (N9)
                 (N10) to [out=-90,in=-130] (N11)
                 (N11)--(N4)
                 
			 (N12) to [out=-160,in=-140] (N11);
\draw[thick,red,dashed] (N9) to [out=-45,in=-130] (N10)
                         (N12) to [out=180,in=-160] (N14);
\draw[yellow!95!black,very thick] (0,\rad) arc (90:20:\rad);
\draw[yellow!95!black,very thick] (0,\rad) arc (90:220:\rad);
\end{tikzpicture}
\caption{The elements of $(C+\overline{J})^0$ characterized by the rise $[(x+1)-b]$ are of the form $y\lhd x$ for $y$ a vertex in the thick yellow part. The edges in red are rises and the edges in blue are descents.}\label{cj0}
 \end{figure}
\begin{lemma}\label{c+j-r}
Let $R\subset J$. We identify the relations in $R$ with the corresponding edges in $\theta(I)$. Then,
\[
(C+\overline{J-R})^0 = (C+\overline{J})^0 + \{ j_f\lhd i_f\ ;\ f\in R\}. 
\]
\end{lemma} 
 \begin{proof}

 Let $f\in R$. By hypothesis $f$ is a descent, so we are in the middle case of Figure \ref{descente-montee}. If we assume that $j_f\lhd i_f \notin (C+\overline{J-R})^0$, then there is $1\leq e \leq i_f$ and $i_f < s \leq j_f$ such that $[e-s]$ is an edge in $\theta(I)-R$. Looking at Figure \ref{descente-montee}, we see that such an edge contradicts the fact that $\theta(I)$ is a noncrossing tree. 
 
 Since it is clear that $(C+\overline{J})^0 \subseteq  (C+\overline{J-R})^0$ and since $(C+\overline{J})^0$ is an interval-poset, the right hand side is a subset of the left hand side. 
 
 Let $b\lhd x\in (C+\overline{J-R})^0-(C+\overline{J})^0$. We assume that $b$ is minimal such that $b\lhd x\notin (C+\overline{J})^0$. We start by stating the main ingredients of the proof which can be easily deduced from the graphics. 
 
 \begin{enumerate}
 \item The minimality of $b$ implies that there is an edge $f=[e-b]\in R$ with $1\leqslant e \leqslant x$.
 \item By hypothesis $f$ is a descent and we are in the leftmost case of Figure \ref{descente-montee} where $i_f = x$. We denote by $[a_1-b_1]$ the first edge that separates $f$ from the basis.
 \end{enumerate}
 
  It is now easy to see that $j_f\lhd i_f\in (C+\overline{J-R})^0-(C+\overline{J})^0$ and $b_1-1\lhd j_f \in (C+\overline{J})^0$ and $b_1-1 \lhd y \notin (C+\overline{J})^0$ for $i_f\leqslant y< j_f$. 
  
  If the edge $[a_1-b_1]$ is not in $R$, then $b_1\lhd i_f\notin (C+\overline{J-R})^0$. And all the relations of $(C+\overline{J-R})^0$ landing at $i_f$ are obtained by taking interval-posets condition and transitivity in the set $\{j_f\lhd i_f\}\sqcup (C+\overline{J})^0$.  Otherwise, $b_1$ is minimal for the property that $b_1 \lhd j_f \in (C+\overline{J-R})^0 - (C+\overline{J})^0$. By induction we see that any relation in $(C+\overline{J-R})^0$ is obtained from $\{j_f\lhd i_f\ ;\ f\in R\}$ and $(C+\overline{J})^0$ by taking the transitive closure and the interval-poset closure.


 \end{proof}

Let $\Delta = (C+\overline{J})^{0}$. Let $\Gamma$ be the smallest interval-poset containing $\{i_f\lhd j_f\ ;\ f\in J\}$.
\begin{lemma}
The relations of $\Delta$ are compatible with the relations of $\Gamma$.
\end{lemma} 
\begin{proof}
The elements of $\Gamma$ are obtained by the interval-poset condition and transitivity from the elements of $\{ i_f \lhd j_f\}$. A typical element is of the form $x\lhd j_f \lhd j_h$ for $f,h\in J$ and $i_f \leqslant x < j_f$. 

If $j_h\lhd x\in \Delta$, then by the interval-poset condition $j_f\lhd x \in \Delta$. If $x\leqslant b$, the edge $[a-b]$ contradicts this fact. If $b<x\leqslant j_f$, by definition of $j_f$ (see the leftmost case of Figure \ref{descente-montee}) there is a descent $[d_2-d_1]$ such that $d_1 < x \leqslant d_2$ which contradicts $b_1-1 \lhd x\in j_f$.  
\end{proof}
As a consequence, one can consider the interval-poset $I_s$ having $\Delta$ as set of decreasing relations and $\Gamma$ as set of increasing relations. Using the two bijections of Section \ref{bij_ip}, we have a better understanding of it. 
\begin{proposition}\label{prop_serre_obscure}
Let $I$ be an exceptional interval-poset of size $n$. Then $I_s = \psi^{-1}\big(\theta(I)\big)$.
\end{proposition}
\begin{proof}
We denote the exceptional interval-poset $\psi^{-1}\big(\theta(I)\big)$ by $I'$. 

By Lemma \ref{6.1}, the increasing cover relations of $I'$ are of the form $i_f \lhd j_f$ for $f$ a descent of $\theta(I)$ which is nothing but a cover decreasing relation in $I$. The decreasing cover relations of $I'$ are of the form $i_f\lhd j_f$ for $f$ a rise of $\theta(I)$ such that $j_f\neq n+1$. This means that $f = a\lhd b$ is an increasing cover relation of $I$ such that $1$ is not smaller than $a$ for $\lhd$. Then, looking at the middle case of Figure \ref{descente-montee} we see that there is a list of rises from $j_f+1$ to $a$. Considering the first of these rises and Lemma \ref{char+delta}, we see that $(b_1-1) \lhd j_f \in (C+\overline{J})^0$. By interval-poset condition we have $i_f \lhd j_f \in (C+\overline{J})^0$. Since $I_s$ is an interval-poset, it is stable by transitivity and we have that all the relations of $I'$ are in $I$.

Conversely, we already saw that the relations $i_f\lhd j_f$ where $f$ runs through the cover decreasing relations of $I$ are in $I'$.

By Lemma \ref{char+delta}, the maximal relations of $(C+\overline{J})^0$ are of the form $(b_1-1) \lhd j_f$ for $f$ an increasing relation $a\lhd b$ such that $1$ is not smaller than $a$ for $\lhd$. Looking at Figure \ref{cj0}, we see that the edge $[a_1-b_1]$ is labeled by $j_f$. The edge labeled by $b_1-1$ is in the subtree delimited by the rise landing at $b_1$ in the succession of rises from $b$ to $b_1$. It is clear that $[a_1-b_1]$ is between this edge and the base. So in $I'$ we have $b_1-1\lhd j_f$. Since $I'$ is an interval-poset, it is stable under interval-poset condition. So $(C+\overline{J})^0\subset I$, and the result follows.   
\end{proof}

\begin{proposition}\label{serre_obscure}
Let $I$ be an exceptional interval-poset with set of decreasing relations in its Hasse diagram denoted by $J$. Let $n_{I} = |J|$. 

Then, in $D^b(\Tam_n)$ we have

\[ \Se(I)\cong \psi^{-1}\big(\theta(I)\big) [n_I].\]

\end{proposition} 
\begin{proof}
By Lemma \ref{6.1} the increasing cover relations of $I_s = \psi^{-1}\big(\theta(I)\big)$ are of the form $i_f \lhd j_f$ for $f$ a descent in $\theta(I)$. In other terms, they are of the form $i_f\lhd j_f$ for $f\in J$. 

Using a dual version of Lemma \ref{reso_proj}, we see that the following complex is an injective resolution of $I_s$: 

\[
{\small
\mathcal{I} = I_{\Delta} \to \bigoplus_{f\in J} I_{\Delta+\{j_f\lhd i_f\}} \to \bigoplus_{\substack{R\subset J \\ |R|=2}} I_{\Delta + \{j_f \lhd i_f\ ;\ f\in R\}}\to \cdots \to \bigoplus_{\substack{R\subset J \\ |R|=|J|-1}} I_{\Delta + \{j_f \lhd i_f\ ;\ f\in R\}}\to I_{\Delta + \{j_f \lhd i_f\ ;\ f\in J\}}.}
\]

On the other hand, by Lemma \ref{reso_proj}, $I$ is isomorphic to $\mathcal{P}(I)$ in $D^b(\Tam_n)$. By Proposition \ref{serre_rong}, the Nakayama functor sends the projective interval-poset $P_C$ to the injective interval-poset $I_{C^0}$. So, by Lemma \ref{c+j-r}, we see that $\Se(I)$ is isomorphic to $\mathcal{I}[n_I]$.
\end{proof}

\section{Duality of noncrossing trees}\label{duality}

Let $T$ be a noncrossing tree in $\mathcal{C}_n$. A noncrossing tree gives a partition of the disk $\mathcal{C}_n$ into $n+1$ areas, that we call the \emph{cells} of the noncrossing tree. Since there is no loop made of edges, each cell has exactly one open side. We label each cell by the leftmost vertex of the open side. The cell containing the vertices $1$ and $n+1$ is called the base cell and is labeled by $n+1$.  

We define the tree $T^{*}$ as the planar dual of $T$. That is, there is an edge $[i,j]$ in $T^{*}$ if and only if the cells in $T$ labeled by $i$ and $j$ are adjacent. 

For $n\in \mathbb{N}$, the rotation by an angle of $\frac{2\pi}{n+1}$ induces a permutation of the noncrossing tree of size $n$. For simplicity, we denote it by $\vartheta$. 

\begin{lemma}\label{double_dual}
Let $T$ be a noncrossing tree. 
\begin{enumerate}
\item The duality commutes with the rotation. That is $\vartheta(T)^* = \vartheta(T^*)$. 
\item The square of the duality is the rotation by an angle of $\frac{2\pi}{n+1}$. That is $(T^*)^* =  \vartheta(T)$.
\end{enumerate}
\end{lemma}
\begin{proof}
Both of the statements are graphical evidences when one draws the planar dual of $T$ on the same circle by adding new vertices in another color between the vertices of $T$ and draws an arrow between them if they are separated by a unique edge of $T$. 
\end{proof}

\begin{proposition}\label{psironddual}
Let $I$ be an exceptional interval poset, then
\[ 
\psi(I)^* = \theta(I).
\]
\end{proposition}
\begin{proof}
This is a consequence of Lemma \ref{6.1}: 

If $x \lhd y$ is a decreasing cover relation in $I$, then by Lemma \ref{6.1} there is a rise $f$ in $\psi(I)$ such that $y=j_f$ and $i_f=x$. Looking at the left case of Figure \ref{descente-montee}, we see that it separates a cell labeled by $i_f$ and a cell labeled by $j_f$. So, the edge $[x-y]$ is in $\psi(I)^{*}$.

If $x\lhd y$ is an increasing cover relation in $I$, then by Lemma \ref{6.1} there is a descent $f$ in $\psi(I)$ such that $x=i_f$ and $y=j_f$. Looking at the middle case of Figure \ref{descente-montee} we see that the edge $[x-y]$ is in $\psi(I)^{*}$. 

Similarly, if $i$ is a maximal point of $I$, it labels a rise $f$ such that $j_f=n+1$. Looking at the right case of Figure \ref{descente-montee}, we see that the edge $[i-(n+1)]$ is in $\psi(I)^{*}$.

In other words every edge of $\theta(I)$, which is not assumed to be a noncrossing tree for this proof, is in $\psi(I)^{*}$.

All the noncrossing trees in $\mathcal{C}_n$ have exactly $n$ edges. So, the tree $\psi(I)^*$ has as many edges as the tree $\psi(I)$. By Lemma \ref{6.1}, the set of edges of $\psi(I)$ is in bijection with the set consisting of the cover relations of $I$ and its maximal elements. In other terms, the tree $\psi(I)^*$ has as many edges as $\theta(I)$. In conclusion, we have $\theta(I) = \psi(I)^*$.

%
%
%
%

\end{proof}
\section{The bounded derived category of the Tamari lattices are fractionaly Calabi-Yau}\label{main_result}
Combining the results of Sections \ref{se_fun} and \ref{duality} we have the following result. 

\begin{theorem}\label{last}
Let $T$ be a noncrossing tree of size $n$. Then, there is an integer $n_T$ such that 
\[\Se(\psi^{-1}(T)) \cong \psi^{-1}(T^*)[n_T] \]
in $D^b(\Tam_n)$.
\end{theorem}
\begin{proof}
By Proposition \ref{prop_serre_obscure} we have that $\Se(\psi^{-1} T) \cong \psi^{-1} (\theta \psi^{-1}T) [n_T]$ for an integer $n_T$. By Proposition \ref{psironddual}, we have $\theta\circ \psi^{-1}T = T^*$. 
\end{proof}
In other terms, the Serre functor acts - up to a shift - as the planar duality of the noncrossing trees. 

It remains to understand the shifts that appear here.
\begin{proposition}\label{shifts}
Let $I$ be an exceptional interval-poset of size $n$. Then,
\[ 
\Se^{2n+2}(I)\cong I[n(n-1)]
\]
in $D^b(\Tam_n)$.
\end{proposition}
\begin{proof}
We denote by $T$ the noncrossing tree $\psi(I)$. By Theorem \ref{last}, the interval-poset obtained after $2n+2$ applications of the Serre functor corresponds via $\psi$ to the tree obtained by successively taking $2n+2$ times the planar duality of $T$. By Lemma \ref{double_dual}, this tree is nothing but $T$. 

We denote by $m_T = n_T + n_{T^*}$. This is the shift obtained after two applications of the Serre functor on $I$. The shift obtained after $2n+2$ applications of the Serre functor is $\sum_{i=0}^{n}m_{\vartheta^{i}T}$. 

By Proposition \ref{serre_obscure}, the number $n_T$ is the number of decreasing relations of $I$ which is also the number of descents of $\theta I = T^{*}$. By Lemma \ref{6.1} the number $n_{T^*}$ is the number of rises $f$ of $T^{*}$ such that $j_f \neq n+1$. Let us call \emph{forbidden} a rise $f$ such that $j_f = n+1$. Then, the number $m_T$ is the number of edges of $T^{*}$ minus the number of forbidden edges. If we denote the number of forbidden edges by $F_{T^*}$, we have
\[ \sum_{i=0}^{n}m_{\vartheta^{i}T} = n(n+1) - \sum_{i=0}^{n}F_{\vartheta^i T^*}. \]

Let $f$ be an edge in $T$. Then it becomes a forbidden edge exactly twice during all the rotations of $T$. This is when $i_f$ or $j_f$ is rotated to $n+1$.

So, we have
\[\sum_{i=0}^{n}m_{\vartheta^{i}T} = n(n+1) -2n = n(n-1). \]
\end{proof}
\begin{theorem}
Let $n\in \mathbb{N}$. The bounded derived category $D^b(\Tam_n)$ is $\frac{n(n-1)}{2n+2}$ fractionally Calabi-Yau.
\end{theorem}
\begin{proof}
The result follows from Proposition \ref{shifts} and Theorem \ref{reduction} and the fact that the projective indecomposable modules are in the family of exceptional interval-posets.
\end{proof}
\begin{remark}
Since the square of the Serre functor acts as the rotation of noncrossing trees, we see that $\frac{n(n-1)}{2n+2}$ is the `effective' dimension of the derived category in the sense that it cannot be simplified when $n
\geqslant 3$. The fraction $\frac{n(n-1)}{2n+2}$ is larger than $1$ when $n\geqslant 4$, this shows that the incidence algebra of the Tamari lattices are not piecewise hereditary. Similarly to the examples of Section $7$ of \cite{lenzing} the Tamari lattices are wild when their Calabi-Yau dimension is larger than $1$. 
\end{remark}
For the Coxeter matrix of the Tamari lattices, we recover Chapton's Theorem (\cite{chapoton_coxeter_tamari}) and we deduce some new properties.
\begin{proposition}\label{cox_mat_new}
Let $n\in \mathbb{N}$. Let $C = -I\cdot (I^{-1})^t$ be the Coxeter matrix of $\Tam_n$.
\begin{enumerate}
\item $C^{2n+2} = \operatorname{Id}$.
\item The elements of the matrix $C$ are $0,1$ and $-1$.
\item The nonzero elements of a column have the same sign. 
\end{enumerate} 
\end{proposition}
\begin{proof}
Since $C$ is up to sign the matrix in the bases of the simple induced by the Serre functor on the Grothendieck group of $D^b(\Tam_n)$ the first point is clear. The image of a simple module by the Serre functor is up to a shift an interval. This implies the two last statements.

\end{proof}
The Tamari lattice $\Tam_n$ is isomorphic to the poset of tilting modules (resp. Cluster tilting modules) over an equioriented quiver of type $A_n$ (resp. $A_{n-1}$). Using results of Ladkani, our result can be generalized to these two families of posets.
\begin{corollary}
The Cambrian lattices of type $A$ and the lattices of tilting modules over a quiver of type $A$ have fractionally Calabi-Yau derived categories. 
\end{corollary}
\begin{proof}
It has been proved by Ladkani in \cite{ladkani_clust_tilt,ladkani_tilt} that these posets are derived equivalent to the Tamari lattices.
\end{proof}

\end{document}